\documentclass[11pt,oneside]{amsart}

\usepackage{amsmath,amssymb}

\newtheorem{thm}{Theorem}[section]
\newtheorem{lem}[thm]{Lemma}

\newtheorem{prop}[thm]{Proposition}
\newtheorem{cor}[thm]{Corollary}
\newtheorem*{cor*}{Corollary \ref{MainCor}}
\newtheorem*{thm*}{Theorem \ref{ThmUnique}}
\theoremstyle{definition}
\newtheorem{claim}[thm]{Claim}
\newtheorem{defn}[thm]{Definition}

\DeclareMathOperator{\Aut}{Aut}
\DeclareMathOperator{\Rist}{Rist}
\DeclareMathOperator{\Stab}{Stab}
\DeclareMathOperator{\Fix}{Fix}
\DeclareMathOperator{\ord}{ord}
\newcommand*{\level}{\mathcal L}
\newcommand*{\Z}{\ensuremath{\mathbf{Z}}}
\newcommand*{\N}{\ensuremath{\mathbf{N}}}
\newcommand*{\fgrbg}{finitely generated regular branch group}
\newcommand*{\GGS}{\textsf{GGS}}

\begin{document}
\title{Weakly maximal subgroups in regular branch groups}
\author[K. Bou-Rabee]{Khalid bou-Rabee}
\thanks{K.B. was supported in part by NSF grant DMS-1405609, P.-H. L. and T. N. were supported by Swiss National Science Foundation.}
\address{School of Mathematics, CCNY CUNY, New York City, New York, USA}
\email{khalid.math@gmail.com}
\author[P.-H. Leemann]{Paul-Henry Leemann}
\address{Department of Mathematics, University of Geneva, Geneva, Switzerland}
\email{Paul-Henry.Leemann@unige.ch}
\author[T. Nagnibeda]{Tatiana Nagnibeda}
\address{Department of Mathematics, University of Geneva, Geneva, Switzerland}
\email{tatiana.smirnova-nagnibeda@unige.ch}
\date{\today}

\begin{abstract}
Let $G$ be a \fgrbg{} acting by automorphisms on a regular rooted tree $T$. It is well-known that stabilizers of infinite rays in $T$ (aka parabolic subgroups) are weakly maximal subgroups in $G$, that is, maximal among subgroups of infinite index. We show that, given a finite subgroup $Q\leq G$, $G$ possesses uncountably many automorphism equivalence classes of weakly maximal subgroups containing $Q$.  In particular, for Grigorchuk-Gupta-Sidki type groups this implies that they have uncountably many automorphism equivalence classes of weakly maximal subgroups that are not parabolic.
\end{abstract}
\maketitle
\setcounter{tocdepth}{2}


\section{Introduction}\label{sec:introduction}

Let $T$ be a locally finite regular rooted tree.
Among groups that act on $T$ by automorphisms, branch groups are of particular interest (see Section~\ref{sec:preliminaries} for all the relevant definitions).
One such example is the Grigorchuk group~$\Gamma$, the first example of a finitely generated group of intermediate growth \cite{GriFirst}.
Branch groups have interesting subgroup structure.
For example, Pervova showed \cite{Pervova}  that all maximal subgroups of $\Gamma$ are of finite index  (there are seven of them, all of index $2$). Pervova's result extends to groups abstractly commensurable with $\Gamma$ \cite{MR2009443}, as well as to other examples of branch groups, such as  Gupta-Sidki groups $G_p$ \cite{Pervova05} and groups abstractly commensurable with them \cite{Garrido}, or multi-edge spinal groups \cite{MultiEdge}; but not to all branch groups (see \cite{Bondarenko}). 
In view of this rigidity of maximal subgroups, the next step in understanding the subgroup structure of branch groups is to study \emph{weakly  maximal subgroups}, that is the maximal elements among the subgroups of infinite index.
Observe that in a finitely generated group every subgroup of infinite index is contained in a weakly maximal subgroup.
Weakly maximal subgroups of regular branch group are the main topic of this note.

Before stating our results, we review some natural examples of weakly maximal subgroups.
Let $G$ be a group of automorphisms of a locally finite regular rooted tree $T$.
The action of $G$ on the tree $T$ extends to an action of $G$ on the boundary $\partial T$ of the tree, by homeomorphisms.
Stabilizers of boundary points are called \emph{parabolic subgroups}.
If $G$ acts on $T$ in a weakly branched way, then all parabolic subgroups are infinite and pairwise distinct \cite{MR2893544}.
Moreover, if the action of $G$ is branched, then all parabolic subgroups are weakly maximal \cite{MR1841750,MR1899368}.
In particular, a branch group has uncountably many conjugacy classes of weakly maximal subgroups. If the group is finitely generated, then it has uncountably many automorphism equivalence classes of weakly maximal subgroups.
In the Grigorchuk group $\Gamma$, the class of weakly maximal subgroups is not reduced to the class of parabolic subgroups:
some sporadic examples were constructed \cite{Solved,MR2893544}.
There are however no classification results for weakly maximal subgroups of $\Gamma$ or other (weakly) branch groups (Problem 6.3 in \cite{Solved}).

Our main purpose in this paper is to show that a finitely generated regular branch group (subject to a minor technical condition) contains uncountably many non parabolic weakly maximal subgroups up to automorphism equivalence.
More precisely, we show that any finite subgroup $Q\leq G$ is contained in uncountably many weakly maximal subgroups.
\begin{thm} \label{MainThm}
Let $T$ be a regular rooted tree and $G\leq \Aut(T)$ be a \fgrbg.
Then, for any finite subgroup $Q\leq G$ there exist uncountably many automorphism equivalence classes
 of weakly maximal subgroups of $G$ containing $Q$.
\end{thm}

As an immediate corollary of the theorem, we have the following result. It indicates that a full classification of weakly maximal subgroups must involve, in a significant way, subgroups that are not parabolic.

\begin{cor} \label{MainCor}
Let $T$ be a regular rooted tree and $G\leq \Aut(T)$ be a \fgrbg.
Suppose that $G$ contains a finite subgroup $Q$ that does not fix any point in $\partial T$.
Then there exist uncountably many automorphism equivalence classes of weakly maximal subgroups of $G$, all distinct from  classes of  parabolic subgroups associated with the action of $G$ on $T$.
\end{cor}

These results hold in a large class of groups. Indeed, all known examples of self-similar branch groups are regular branch.
Well-known examples include the first Grigorchuk group $\Gamma$ and Gupta-Sidki groups \cite{MR2035113} and many of the \GGS{} groups (see Definition \ref{DefGGS} and the discussion below).
Moreover, all these examples essentially admit a unique branch action on a spherically regular tree \cite{MR2011117}.
(An infinite rooted tree is \emph{spherically regular} if the degree of a vertex depends only on its distance from the root.)
The unicity of branch action, together with our Theorem \ref{MainThm} allow us to deduce the following.

\begin{thm}\label{ThmUnique}
Let $G$ be either the first Grigorchuk group, or a \GGS{} group with $E=(1,-1)$ or $E=(\epsilon_1,\dots,\epsilon_{p-1})$, $p$ prime, such that the $\epsilon_i$ are not all zero and not all non-zero.
Then $G$ has uncountably many automorphism equivalence classes of weakly maximal subgroups, all distinct from classes of parabolic subgroups of any branch action of $G$ on a spherically regular tree.
\end{thm}

We also demonstrate that, loosely speaking, weakly maximal subgroups of $\Gamma$ and \GGS{} groups live as deep in the tree as one desires.

\begin{thm} \label{MainTechnicalResult}
Let $T$ be a regular rooted tree and $G\leq \Aut(T)$ be a \fgrbg.
Suppose that $G$ contains an infinite index subgroup $Q$ that does not fix any point in $\partial T$.
Then there exists an integer $l$ such that for any $k \in \N$ there exists a weakly maximal subgroup of $G$ which stabilizes the $k$\textsuperscript{th} level and does not stabilize any vertex of the $(k+l)$\textsuperscript{th} level.

More precisely, we can choose $l$ to be the minimal integer such that there exists an infinite index subgroup $Q$ that does not stabilize any vertex of the $l$\textsuperscript{th} level.

If $G$ is the first Grigorchuk group or a regular branch \GGS{} group, then it is possible to take $l=1$.
\end{thm}

This paper is organized as follows.
In \S \ref{sec:preliminaries}, we fix notation and prove some results about weakly maximal subgroups. 
\S \ref{sec:proof} contains the proof of Theorem \ref{MainTechnicalResult} and \S \ref{sec:proof1}  the proofs of Theorem \ref{MainThm}, Corollary \ref{MainCor} and Theorem \ref{ThmUnique}.

We are grateful to Slava Grigorchuk for numerous valuable remarks on the first version of the paper. We are also grateful to an anonymous referee for corrections that improved the paper.
\section{Preliminaries}\label{sec:preliminaries}
\subsection{Regular branch groups and just infinite groups}
For the $d$-regular rooted tree  $T$  denote by $\level_n$ the set of vertices of $T$ \lq\lq at level $n$\rq\rq , i.e., at distance $n$ from the root.
We picture the  tree $T$ with the root at the top, with $d$ edges connecting every vertex of $\level_n$ with vertices of $\level_{n+1}$.
This gives a natural (partial) order on vertices of $T$ defined by $w\leq v$ if $w$ lies on an infinite ray emanating from $v$.
For a vertex $v$ in $T$, we note $T_v$ the subtree of $T$ consisting of all vertices $w\leq v$.

Let $G\leq \Aut(T)$ be a group of automorphisms of $T$. It is said to be \emph{spherically transitive} if it acts transitively on each level of $T$.
We denote by $\Stab_G(v)$ the stabilizer of a vertex $v$ and by $\Stab_G(n)=\cap_{v\in \level_n} \Stab_G(v)$ the stabilizer of the level $n$.
We therefore have injective maps 
\begin{align*}
	\psi_v\colon\Stab_G(v)&\to\Aut(T)			&\psi_1\colon \Stab_G(1)&\to \Aut(T)^d\\
					\phi&\to\phi|_{T_v}		&\phi&\mapsto (\phi|_{T_{v_1}},\dots,\phi|_{T_{v_d}})
\end{align*}
where $v$ is any vertex of $T$, $v_1,\dots ,v_d$ are the vertices of the first level and $T_{w}$ denotes the subtree of $T$ growing from the vertex $w$.
For a $d$-regular rooted tree $T$, a group $G\leq \Aut(T)$ is said to be \emph{self-similar} if for any $v$, $\psi_v(\Stab_G(v))$ is equal to $G$, after after an identification of $T$ with $T_v$.

The \emph{rigid stabilizer} $\Rist_G(v)$ of a vertex $v$ in $G$ is the set of elements of $G$ acting trivially outside $T_v$.
The rigid stabilizer of a level $\Rist_G(n)$ is the subgroup generated by the $\Rist_G(v)$ for all vertices $v$ of level $n$.

A spherically transitive group $G\leq \Aut(T)$ is \emph{weakly branch} if all rigid stabilizers of levels are infinite and \emph{branch} if all rigid stabilizers of levels are of finite index.

\begin{defn}
For a $d$-regular rooted tree $T$, a group $G\leq \Aut(T)$ is said to be \emph{regular branch (over $K$)} if it satisfies the following conditions
\begin{itemize}
\item G is spherically transitive;
\item $G$ is self-similar;
\item there exists a finite index subgroup $K$ of $G$ such that $K^d$ is contained in $\psi_1(K\cap \Stab_G(1))$ as a subgroup of finite index.
\end{itemize}
\end{defn}

Since $G$ is self-similar, the image of $\psi_1$ is included in $G^d$ and we can therefore iterate $\psi_1$ to have a sequence
\[
	\psi_n\colon\Stab_G(n)\hookrightarrow G^{d^n}
\]
It follows from the definition that $K^{d^n}$ is contained in $\psi_n(\Stab_G(n))$ as a subgroup of finite index and therefore that for all level $n$, $\Rist_G(n)$ is of finite index in $G$ \cite{MR2893544}.
This proves that a regular branch group is a branch group, and it follows directly from definitions that a branch group is weakly branch.
Since $K^{d^n}\leq\psi_n(\Stab_G(n))\leq G^{d^n}$, the image of $\psi_n$ is of finite index in $G^{d^n}$.

A group $G$ is \emph{just-infinite} if all its proper normal subgroups are of finite index (equivalently if all its non-trivial quotient are finite).

\begin{lem}\label{lem:elconj}
Let $G$ be a just-infinite group, $H$ be an infinite index subgroup of $G$ and $\gamma\in G$ be any nontrivial element.
Then there is a conjugate of $\gamma$ that is not in $H$.
\end{lem}
\begin{proof}
Suppose that $H$ contains all the conjugates of $\gamma$.
Then $H$ contains $\langle \gamma\rangle^G$, the normal closure of $\gamma$ in $\Gamma$.
This subgroup is normal and therefore of finite index since $G$ is just-infinite.
But this contradicts the fact that $H$ is of infinite index.
\end{proof}

\begin{prop} \label{prop:conj}
Let $G$ be any group and $H$ a weakly maximal subgroup.
If $\gamma \in G \setminus H$ has finite order, then $\gamma H \gamma^{-1} \neq H$.

In particular, in a torsion group $G$, any weakly maximal subgroup is self-normalizing.
\end{prop}

\begin{proof}
The proof is by contradiction.
Suppose that $\gamma H \gamma^{-1} = H$ and $\gamma \notin H$ has order $k$.
Then $\left< \gamma, H \right>$ is a homomorphic  image of $H \rtimes \langle\gamma\rangle$.
Therefore, the group $\left< \gamma, H \right>$ is a homomorphic image of $H \rtimes \Z/k\Z$ and thus $H$ has finite index in $\left< \gamma, H \right>$.
Since $H$ is a weakly maximal subgroup of $G$, the group $\left< \gamma, H \right>$ is of finite index in $G$, and thus 
$H$ is of finite index in $G$, which contradicts $H$ being weakly maximal.
\end{proof}

\begin{cor} \label{cor:infinitelymanyconj}
Let $G$ be a just-infinite group and let 
$s=\sup\{p^n\,|\, p\textnormal{ prime, }n\in\mathbf N,\exists\gamma\in G \textnormal{ of order }p^n\}$.
Then every conjugacy class of subgroups containing a weakly maximal subgroup of $G$ contains at least $s$ subgroups.
\end{cor}

\begin{proof}
If $s=0$, the conclusion is trivially true.
%
If $s>0$, take $\gamma$ a non-trivial element of order $p^n$.
If $n=1$, by Lemma \ref{lem:elconj}, given a weakly maximal subgroup $H$, there exists $f\gamma f^{-1}$ a conjugate of $\gamma$ which is outside $H$.
Since $p$ is prime, for all $0<i<p$, the element $f\gamma^i f^{-1}$ is not in $H$.
If $n\geq 2$, we apply Lemma \ref{lem:elconj} to $\gamma^{p^{n-1}}$ to find $f\in G$ such that $h=f\gamma^{p^{n-1}}f^{-1}$ is outside $H$.
Since for all $1<i<p^n$, $h$ belongs to the subgroup generated by $f\gamma^{{i}}f^{-1}$, none of the $f\gamma^{{i}}f^{-1}$ are in $H$.

By Proposition \ref{prop:conj}, for each $f\gamma^if^{-1} \neq f\gamma^jf^{-1}$, we have 
\[
	f\gamma^if^{-1} H f\gamma^{-i}f^{-1} \neq f\gamma^{j}f^{-1} H f\gamma^{-j}f^{-1}.
\]
Thus the conjugacy class of $H$ contains at least $\ord(\gamma)$ elements.
\end{proof}

\subsection{Grigorchuk group and \GGS{} groups} \label{sec:grig}

An important example of a \fgrbg{} is the first Grigorchuk group $\Gamma$. It is a $2$-group and it has word growth strictly between polynomial and exponential \cite{GriFirst}. We refer the reader to \cite[Chapter VIII]{MR1786869} for a detailed introduction into the group $\Gamma$ and for proofs of the properties of $\Gamma$ mentioned in this Section.

The group $\Gamma$ can be defined as a subgroup of the group of automorphisms of the infinite binary tree, $\Gamma:=\left< a, b, c, d \right>$, generated by four automorphisms $a,b,c,d$ defined recursively, as follows
\begin{align*}
a(xw) &= (\overline{x}w)&&  \\
b(0w) &=  0a(w) & b(1w) &=  1c(w) \\
c(0w) &=  0a(w) & c(1w) &=  1d(w) \\
d(0w) &=  0w &(1w) &=  1b(w) \\
\end{align*}
where $x\in\{0,1\}$; $w$ denotes an arbitrary binary word; $\overline{1} = 0$ and $\overline{0} = 1$.

In the case when $g \in \Aut(T)$ fixes the first $k$ levels of the tree, we will say that $g$ has \emph{level $k$} and we will write
$g = (\gamma_1, \ldots, \gamma_{2^k})_k$ in order to record the action beyond level $k$ only.
For example, $b = (a, c)_1$, $c = (a, d)_1$, and $d = (1, b)_1$ all have level 1.

The action of $\Gamma$ on the infinite binary tree is branched, and it can be shown moreover that $\Gamma$ is regular branch over the normal closure of the element $(ab)^2$, see \cite{MR2035113}.
This can be used to prove that $\Gamma$ is just-infinite.

While the next lemma is not used in the proofs, we record it for possible future applications:

\begin{lem} \label{lem:topological}
Let $H$ be a weakly maximal subgroup of $\Gamma$.
Let $\overline{H}$ denote the closure of $H$ in the profinite completion $\hat \Gamma$ of $\Gamma$.
Then $\overline{H} \cap \Gamma = H$.
\end{lem}

\begin{proof}
Suppose not, then there exists $h \in (\overline{H} \cap \Gamma) \setminus H$.
Since $H$ is weakly maximal, the group $M = \left< H, h \right>$ is of finite index in $\Gamma$.
Thus, by Schreier's rank theorem, $M$ is finitely generated.
The subgroup $[M,M]$ of $M$ is characteristic and since $M$ is torsion, $[M,M]$ has finite index in $M$ and hence in $\Gamma$.
Let $N$ be the normal core of $[M,M]$. Since $\Gamma$ is just-infinite, we have that $N$ is of finite index in $\Gamma$.
We have that $H N = M N$, since $h N \leq H N$, thus we have that $H$ contains finitely many elements $a_1, \ldots, a_n$ that generate $M/N$ and hence generate $M/[M,M]$.
Further, the set $\{ a_i \}$ generates any nilpotent quotient of $M$, as in a nilpotent group the Frattini subgroup always contains the derived subgroup \cite[Lemma 5.9, page 350]{MR2109550}.
Since every proper quotient of the Grigorchuk group is nilpotent and $\Gamma$ is subgroup separable \cite{MR2009443}, we must have $H = M$.
Then $h \in H$ --- a contradiction.
\end{proof}

Other interesting examples of groups acting on a regular tree, like the second Grigorchuk group \cite{GriFirst} or Gupta-Sidki groups \cite{GuptaSidki} arise as \GGS{} groups (Grigorchuk-Gupta-Sidki groups, the terminology comes from \cite{Bau93}).
For a general treatment and more examples, see \cite{BarThese}, \cite{MR2035113} and \cite{Vovkivsky} for the case of a $p^n$-regular tree, $p$ prime.
\begin{defn}\label{DefGGS}
Let $T$ be a $d$ regular rooted tree and $a=((1,2,\dots,d))_0$ be the automorphism permuting cyclically the edges coming from the root. For any vector $E=(\epsilon_1,\dots,\epsilon_{d-1})$, with $0\leq \epsilon_i<d$ integers, define the associated \emph{\GGS{} group} $G_E=\langle a,b\rangle$, where $b=(a^{\epsilon_1},\dots,a^{\epsilon_{d-1}},b)_1$.
The vector $E$ is called the \emph{defining vector} of $G_E$.
\end{defn}
The Gupta-Sidki groups (with $E=(1,-1,0\dots,0)$) and the second Grigorchuk group (with $p=4$ and $E=(1,0,1)$) were the first examples of \GGS{} groups.
%
Suppose now that $p$ is prime and that the numbers $\epsilon_i$ are not all zero and not all non-zero, or that $E=(1,-1)$.
In this case, $G_E$ is regular branch \cite{GGS-groups}, just infinite \cite{Vovkivsky} and essentially admits a unique branch action on a spherically regular rooted tree \cite{MR2011117}.

\section{The proof of Theorem \ref{MainTechnicalResult}} \label{sec:proof}

We first demonstrate that if $Q\leq \Aut(T)$ stabilizes no point in $\partial T$, then there exists an integer $l$ such that $Q$ stabilizes no vertex of level $l$.
The proof is done by contradiction.
Let $T_Q$ be the full subgraph of $T$ consisting of all vertices stabilized by $Q$, with edges coming from edges in $T$.
Assume that for each integer $l$, the group $Q$ stabilizes a vertex $v$ of level $l$.
Then $Q$ stabilizes the path between the root and $v$.
Therefore, $T_Q$ is a locally finite tree with paths of arbitrary big length.
By K\"onig's lemma, $T_Q$ contains an infinite path $\xi$ which is stabilized by $Q$ --- contradiction.

\begin{proof}[Proof of Theorem \ref{MainTechnicalResult}]
Let $d$ be the degree of $T$.
Using $Q$ we will, for any integer $k \geq 1$, produce the desired subgroup $W$ of $\Stab_G(k)$.
For any $k$, let $\psi_k$ be the map 
\[
	\Stab_G(k)  \hookrightarrow G^{d^k}
\]
 and for any $i = 1, \ldots, d^k$, let $\pi_i$ be the projection of $G^{d^k}$ onto the $i$th factor.
Let $\Delta$ be a weakly maximal subgroup of $G$ containing $Q$ and look at $(\pi_1\circ \psi_k)^{-1}( \Delta)$.
Since $G$ is regular branch, the image of $\psi_k$ has finite index in $G^{d^k}$.
This and the surjectivity of $\pi_1$ implies that $(\pi_1\circ \psi_k)^{-1}( \Delta)$ is an infinite index subgroup of $\Stab_G(k)$.
Let $H$ be any weakly maximal subgroup containing it.
We claim that $H \leq \Stab_G(0^k)$, where $0^k$ is the leftmost vertex of level $k$.
Suppose, for the sake of contradiction, that there exists $\gamma=(g_1,\dots,g_{d^k})_k\tau \in H \setminus \Stab_G(v)$.
Then for any $(a_1, \ldots, a_{d^k})_k \in H$, we have $\gamma (a_1, \ldots, a_{d^k})_k \gamma^{-1} \in H$.
Since $\gamma$ is not in $\Stab_G(v)$, we have that
\[
	\gamma (a_1, \ldots, a_{d^k})_k  \gamma^{-1} = (g_ja_jg_j^{-1}, b_2, \ldots, g_1a_1g_1^{-1}, \ldots, b_{d^k})_k
\]
for some $j \in 2, 3, \ldots, d^k$; with $b_l$'s being a permutation of conjugates of the remaining $a_i$'s.
However, since $G$ is regular branch, the image of $\psi_k$ is of finite index in $G^{d^k}$.
Thus, the image of $\psi_k$ contains $Q_1 \times Q_2 \cdots \times Q_{d^k}$, where each $Q_i$ is a finite index subgroup of $G$ for $i = 1, \ldots, d^k$.
It follows, then that $\psi_k\bigl(H \cap \Stab_G(k)\bigr)$ contains $1 \times Q_2 \cdots \times Q_{d^k}$, as $1 \times Q_2 \times \cdots \times Q_{d^k}$ is in the kernel of $\pi_1$.
But then $\gamma (H \cap \Stab_G(k)) \gamma^{-1}=H \cap \Stab_G(k)$ and $\psi_k(\gamma (H \cap \Stab_G(k)) \gamma^{-1})$ contains a conjugate of $Q_j$ in the first factor.
Therefore, $\psi_k(H \cap \Stab_G(k))$ is of finite index in $G^{d^k}$.
It follows, because $\psi_k$ is injective, that $H \cap \Stab_G(k)$ is of finite index in $\Stab_G(k)$, which is impossible as $H$ is weakly maximal and $\Stab_G(k)$ is a subgroup of $G$ of finite index.
So we have found our desired contradiction and proved that $H \leq \Stab_G(0^k)$.
The same argument shows that $H \leq \Stab_G(v)$ for any vertex $v$ of the $k$\textsuperscript{th} level; and hence $H \leq \Stab_G(k)$.

Finally, the map $\pi_i \circ \psi_k$ maps $H$ onto all of $G$ for $i \neq 1$.
And $\pi_1 \circ \psi_k(H)$ contains $Q$.
Thus, it is impossible for $H$ to stabilize any vertex of the $(k+l)$\textsuperscript{th} level, as desired.

When $G$ is the first Grigorchuk group or a \GGS{} group, it is possible to take $Q=\langle a\rangle$.
\end{proof}

We remark here that it is now possible to produce countably many groups each of which is not conjugate to any parabolic subgroup of $\Gamma$.
For this, set $\Fix(H)$ to be the points in $T$ fixed by every element in $H$.
For any $g \in \Gamma$ and $H \leq \Gamma$, recall that
$\Fix(g H g^{-1}) = g \Fix (H)$.
By this relation, the conjugate of any parabolic group is parabolic.
Thus, if $H_k$ is conjugate to a parabolic group, it must be parabolic itself.
However, $H_k$ does not fix any vertex of level $k+1$ and conjugation preserves levels of the tree, so this is impossible.

Next, we show that for $i \neq j$, the groups $H_i$ and $H_j$ are never conjugate.
Suppose, for the sake of contradiction, that $H_i$ and $H_j$ are conjugate and $i < j$.
Then $\Fix(H_i) = g \Fix(H_j)$ for some $g \in \Gamma$.
However, by construction, $H_j$ fixes a vertex, $v$, of level $j$.
It follows then that $H_i$ fixed $gv$, but $H_i$ does not fix any vertex at level $i+1$, and hence cannot fix a vertex at level $j \geq i+1$.
It follows then that $H_i$ and $H_j$ are not conjugate.

\section{The proofs of Theorem \ref{MainThm}, Corollary \ref{MainCor} and Theorem \ref{ThmUnique}} \label{sec:proof1}
Let $T$ be a regular rooted tree, $G\leq \Aut(T)$ be a \fgrbg{} and $Q\leq G$ a finite subgroup of $G$.
We will prove that there are uncountably many weakly maximal subgroups of $G$ containing $Q$.
This is sufficient to prove Theorem \ref{MainThm} since $G$ is finitely generated and therefore $\Aut(G)$ is countable.

Our proof here is similar, in a sense, to Cantor's proof of uncountability of the real numbers.

Suppose, for the sake of contradiction, that there exists countably many non-parabolic weakly maximal subgroups that contain $Q$.
Enumerate them $\{ W_i \}_{i\geq1}$.
Since $Q$ is finite, there exists $k_1$ such that $Q\cap \Stab_G(k_1)=\{1\}$ and $Q$ does not act transitively on $\level_{k_1}$.
If $W_1$ contains $\Rist_G(v)$ for all $v \in \level_{k_1}$, then $W_1$ contains $\Rist_G(k_1)$.
This is impossible as $\Rist_G(k_1)$ is of finite index in $G$ because $G$ is branch.
Hence, there exists some $v_0 \in \level_{k_1}$ such that $\Rist_G(v_0)\setminus W_1$ is non-empty.
Pick an element in this non-empty set and call it $w_1$.

\begin{claim} \label{FirstClaim}
Set $H_1:=\left< Q, w_1 \right>$.
Then there exists a vertex $u$ of level $k_1$ such that the subtrees $\{ T_{q(u)}: q \in Q \}$ are all fixed by $H_1 \cap \Stab_G(k_1)$.
Moreover, $H_1 \cap \Stab_G(k_1)$ is the normal closure of $w_1$ in $H_1$.
\end{claim}
\begin{proof}
Since $w_1$ is in $\Rist_G(v_0)$, for any $q\in Q$, $qw_1q^{-1}$ belongs to $\Rist_G\bigl(q(v_0)\bigr)$.
Hence, for any $g\in H_1$, $gw_1g^{-1}$ fixes $T_v$ for $v\in\level_{k_1}\setminus Q(v_0)$.
By the choice of $k_1$, $Q(v_0)$ cannot contain all vertices of level $k_1$, thus there exists a vertex $u\in\level_{k_1}\setminus Q(v_0)$ such that $\{ T_{q(u)}: q \in Q \}$ are all fixed by the normal closure of $w_1$ in $H_1$.
We conclude by showing that $H_1 \cap \Stab_G(k_1)$ is exactly the normal closure of $w_1$ in $H_1$.
By definition, the element $w_1$, and hence any of its conjugates, are in $H_1 \cap \Stab_G(k_1)$.
Further, every non-trivial element of $Q$ is not in $\Stab_G(k_1)$.
Hence, the quotient of $H_1$ by the normal closure of $w_1$ in $H_1$ is precisely $Q$, and, moreover, it follows that $H_1 \cap \Stab_G(k_1)$ is the normal closure of $w_1$ in $H_1$, as desired.
\end{proof}

For a vertex $w$ and a set $S$ of vertices in $T$, we write $w \leq S$ if $w \leq v$ for some $v \in S$.
Suppose $w_1, \dots, w_i \in G$ of level $k_1,\dots, k_i$ have been constructed so that $H_i:= \left< Q, w_1, \ldots, w_i \right>$ with
\begin{enumerate}
	\item $w_j$ does not belongs to $W_j$ for all $1\leq j\leq i$,

	\item the normal closure of $\{ w_1, \ldots, w_i \}$ in $H_i$ is $H_i \cap \Stab_G(k_1)$, and

	\item for $2\leq j\leq i$, there exist vertices $u_j$ of level $k_j$ with $u_j\leq Q(u_{j-1})$, such that $H_j \cap \Stab_G(k_1)$ fixes every point in $\{ T_{q(u_j)}: q \in Q \}$.
\end{enumerate}

\begin{claim} \label{SecondClaim}
For all $k_{i+1}$ big enough there exists $w_{i+1} \in \Stab_G(k_{i+1})$ such that $H_{i+1}:=\left< Q, w_1, \ldots, w_{i+1} \right>$ is not included in any $W_1, \ldots, W_i, W_{i+1}$, and, further, there exists a vertex $u_{i+1}\leq Q(u_i)$ of level $k_{i+1}$ such that $H_{i+1} \cap \Stab_G(k_1)$ fixes every element in $\{ T_{q(u_{i+1})}: q \in Q \}$.
Moreover, $H_{i+1} \cap \Stab_G(k_1)$ is the normal closure of $w_1, w_2, \ldots, w_{i}, w_{i+1}$ in $H_{i+1}$.
\end{claim}
\begin{proof}
Since $Q$ is finite, it is possible to find $k_{i+1}>k_i$ such that $Q$ does not acts transitively on $\level_{k_{i+1}}\cap\{ T_{q(u_{i})}: q \in Q \}$.
Since $W_{i+1}$ is of infinite index in $G$ and $\Rist_G(k_{i+1})$ is of finite index, there exist an element $w_{i+1}\in\Rist_G(v_0)\setminus W_{i+1}$ for some $v_0 \in \level_{k_{i+1}}$.

From now on, when we write $v$, we will always assume it is a vertex of level $k_{i+1}$.
By assumption, $H_i\cap\Stab_G(k_1)$ acts trivially on $\{ T_{q(u_i)}: q \in Q \}$.
Therefore, for any $p \in H_i\cap \Stab_G(k_1)$, $pw_{i+1}p^{-1}$ fixes $T_v$ for all $v \leq Q(u_i)$ with $v\neq v_0$.
By definition of $k_{i+1}$, there exists $u_{i+1} \leq  Q(u_i)$ of level $k_{i+1}$ with $u_{i+1}\notin Q(v_0)$.
 For such a $u_{i+1}$, every element in $\{T_{q(u_{i+1})}: q\in Q\}$ is fixed by the normal closure of $\{w_1,\dots,w_{i+1}\}$ in $H_{i+1}$.
 In fact, since each of the trees in $\{T_{q(u_{i+1})}: q\in Q\}$ are contained in $\{T_{q(u_{i})}: q\in Q\}$, it follows that for all $q \in Q$ and $j = 1, \ldots, i+1$, $q w_j q^{-1}$ fixes $\{T_{q(u_{i+1})}: q\in Q\}$.
 Set $N$ to be the normal closure of $\{w_1, \ldots, w_{i+1}\}$ in $H_{i+1}$.
Since $N$ is generated by $\{ q w_j q^{-1}: q \in Q \text{ and } j = 1, \ldots, i+1 \}$, it follows that $N$ fixes $\{T_{q(u_{i+1})}: q\in Q\}$.
Since $Q \cap \Stab_G(k_1)$ is trivial and $N \leq \Stab_G(k_1)$, it follows that $N = \Stab_G(k_1) \cap H_{i+1}$, as desired.
\end{proof}

By applying Claims \ref{FirstClaim} and \ref{SecondClaim}, we construct an infinite sequence $k_1<k_2<\dots$ and an infinite sequence $\{ w_i \}_{i\geq1}$ in $G$ such that:

\begin{itemize}
	\item $H_i:= \left< Q, w_1, \ldots, w_i \right>$ is not included in any $W_1, \ldots, W_i$,

	\item there exists a vertex $u_i \in T$ of level $k_i$, such that $H_i \cap \Stab_G(k_1)$ fixes every point in $T_{u_i}$.
\end{itemize}

We further claim that for any $i \in \N$, the associated group $H_i$ has infinite index in $G$.
Suppose, for the sake of contradiction, that $H_i$ is of finite index in $G$ for some $i$.

Note that $H_i\cap \Stab_G(k_1)$ acts trivially on some $T_v$ where $v \in T$ has level $k_i > k_1$.
Since $G$ is regular branch over $K$, $K^{d^{k_i}}$ is a finite index subgroup of $\psi_{k_i}(\Stab_G(k_i)\cap K)$.
This implies that $\psi_{k_i}\bigl(\Stab_G(k_i)\bigr)$ is a finite index subgroup of $G^{d^{k_i}}$.
The subgroups $H_i$ and $\Stab_G(k_i)$ are of finite index in $G$, thus $\Stab_G(k_i)\cap H_i$ has finite index in $\Stab_G(k_i)$.
Since $\psi_{k_i}\bigl(\Stab_G(k_i)\bigr)$ is a finite index subgroup of $G^{d^{k_i}}$, $\psi_{k_i}(\Stab_G(k_i)\cap H_i)$ has finite index in $G^{d^{k_i}}$.
This conclusion is impossible as every element in  $H_i \cap \Stab_G(k_i) \leq H_i \cap \Stab_G(k_1)$ fixes every point in $T_{u_i}$.
The claim is now shown.

Now, given that $H_i$ is never of finite index, we claim that
$H = \cup_{i\geq 1} H_i$
is an infinite-index subgroup of $G$. 
Suppose not, then $H$, being of finite index in a finitely generated group $G$ is finitely generated, say,  by elements $a_1, \ldots, a_m$.
Thus, as the sequence of groups $\{ H_i \}_{i\geq1}$ is increasing, there is a single $H_i$ that must contain each $a_1, \ldots, a_m$.
It follows that $H_i$ is of finite index -- a contradiction.

Since $G$ is finitely generated, there exists a weakly maximal subgroup $A$ containing $H$ in $G$.
Therefore, $A$ contains $Q$ and hence equal to one of the $W_i$.
On the other hand, for any $W_i$, $H$ contains an element that is not in $W_i$, which give us the desired contradiction.
Thus, there cannot be countably many weakly maximal subgroups containing $Q$.
This completes the proof of Theorem~\ref{MainThm}.~\qedsymbol

We now prove Corollary \ref{MainCor} and Theorem \ref{ThmUnique} as direct corollaries of Theorem~\ref{MainThm}.
\begin{cor*}
Let $T$ be a regular rooted tree and $G\leq \Aut(T)$ be a \fgrbg.
Suppose that $G$ contains a finite subgroup $Q$ that does not fix any point in $\partial T$.
Then there exist uncountably many automorphism equivalence classes of weakly maximal subgroups of $G$, all distinct from  classes of  parabolic subgroups associated with the action of $G$ on $T$.
\end{cor*}
\begin{proof}
Let $\phi$ be an automorphism of $G$.
Since $G$ is weakly branch, by Theorem 7.3 of \cite{Lavreniuk} there exists $\psi$ a homeomorphism of $\partial T$ such that for all $g\in G$ and $\xi\in\partial T$ we have $\phi(g)(\xi)=\psi^{-1}\cdot g\cdot\psi(\xi)$.
Therefore, automorphisms of $G$ send parabolic subgroups to parabolic subgroups.
On the other hand, if $W$ is one of the uncountably many weakly maximal subgroups containing $Q$, it fixes no ray and therefore is not parabolic.
\end{proof}
\begin{thm*}
Let $G$ be either the first Grigorchuk group, or a \GGS{} group with $E=(1,-1)$ or $E=(\epsilon_1,\dots,\epsilon_{p-1})$, $p$ prime, such that the $\epsilon_i$ are not all zero and not all non-zero.
Then $G$ has uncountably many automorphism equivalence classes of weakly maximal subgroups, all distinct from classes of parabolic subgroups of any branch action of $G$ on a spherically regular tree.
\end{thm*}
\begin{proof}
Take $Q=\langle a\rangle$. 
Then, there exists uncountably many automorphism equivalence classes of weakly maximal subgroups of $G$ all distinct from classes of parabolic subgroups of the action of $G$ on $T$. The rigidity described in \cite{MR2011117} implies that for any branch action of $G$ on a spherically regular tree, the parabolic subgroups are the same as the ones from the original action.
\end{proof}

%
%

\end{document}